\documentclass[a4paper,12pt]{article}
\usepackage{babel}
\usepackage[utf8]{inputenc}
\usepackage{amsthm}
\usepackage{amsfonts}
\usepackage{amssymb}
\usepackage{graphicx}
\usepackage{rotating}
\usepackage{wrapfig}
\usepackage[all]{xy}
\usepackage{pgf,tikz}
\usetikzlibrary{arrows}
\usetikzlibrary[patterns]

\newtheorem{cor}{Corollary}
\newtheorem{coro}{Corollary}

\newtheorem{defi}{Definition}

\newtheorem{teo}{Theorem}
\newtheorem{theo}{Theorem}

\newtheorem{lema}{Lemma}

\newtheorem*{acknowledgements}{Acknowledgements}

\title{Spaces where all bijections are morphisms}
\author{Lucas H. R. de Souza}

\begin{document}

\DeclareGraphicsExtensions{.pdf,.jpg,.mps,.png,}

\maketitle

\begin{abstract}Here we classify all topological spaces where all bijections to itself are homeomorphisms. As a consequence, we also classify all topological spaces where all maps to itself are continuous.

Analogously, we classify all measurable spaces where all bijections to itself are measurable with measurable inverse. As a consequence, we also classify all measurable spaces where all maps to itself are measurable.
    
\end{abstract}

\let\thefootnote\relax
\let\thefootnote\relax

\section{Introduction}

We have trivially that topological spaces with the trivial or the discrete topology have the property that all bijections to itself are homeomorphisms. Consider the following family of topological spaces:

\begin{defi}Let $\kappa$ be a cardinal number and $X$ a topological space. The co-$\kappa$ topology on $X$ is the topology where the open sets are the empty set and the subsets of $X$ whose complement has cardinality less than $\kappa$. We say that a topological space is cocardinal if it has the co-$\kappa$ topology for some cardinal $\kappa$.    
\end{defi}

This topology is not defined for $1 < \kappa < \aleph_{0}$. Note that the co-$\aleph_{0}$ topology is the cofinite topology (see p. 49 of \cite{SS} for more information about these topological spaces) and the co-$\aleph_{0}^{+}$ topology is the cocountable topology (if $\alpha$ is a cardinal, then we denote by $\alpha^{+}$ its successor cardinal, see Chapters I.10 to I.13 of  \cite{Kun} for cardinal numbers). Note also that the co-$1$ topology is the trivial topology and the co-$\kappa$ topology is the discrete topology if $\kappa$ is bigger than the cardinality of $X$.

It is easy to see that, for cocardinal topologies, all bijective maps are homeomorphisms. Our objective is to show that this property actually characterizes the cocardinal topologies:

\begin{theo}Let $X$ be a topological space. Then $Homeo(X) = bij(X)$ if and only if $X$ is cocardinal.
\end{theo}

Here $Homeo(X)$ is the set of all homeomorphisms from $X$ to $X$ and $bij(X)$ is the set of all bijective maps from $X$ to $X$.

As a consequence, we have the following:

\begin{coro}Let $X$ be a topological space. Then $\mathcal{C}(X,X) = Map(X,X)$ if and only if $X$ has the trivial topology or the discrete topology.    
\end{coro}

Here $\mathcal{C}(X,X)$ is the set of continuous maps from $X$ to $X$ and $Map(X,X)$ is the set of all maps from $X$ to $X$.

We have also an analogous version of this theorem for measurable spaces:

\begin{defi}Let $\kappa$ be a cardinal number and $X$ a measurable space. The $\kappa$-co-$\kappa$ algebra on $X$ is the $\sigma$-algebra where the measurable sets are the subsets of $X$ with cardinality less than $\kappa$ or its complement has cardinality less than $\kappa$. We say that a measurable space is cocardinal if it has the $\kappa$-co-$\kappa$ algebra for some cardinal $\kappa$.    
\end{defi}

This $\sigma$-algebra is not defined for $1 < \kappa \leqslant \aleph_{0}$. Note that the $\aleph_{0}^{+}$-co-$\aleph_{0}^{+}$ algebra is the countable-cocountable algebra (see Examples 3.3 of \cite{Sc} for the proof that this is actually a $\sigma$-algebra). Note also that the $1$-co-$1$ algebra is the trivial $\sigma$-algebra and the $\kappa$-co-$\kappa$ algebra is the discrete $\sigma$-algebra if $\kappa$ is bigger than the cardinality of $X$.

If $X$ is a measurable space, let $Iso(X)$ be the set of all measurable maps from $X$ to $X$ with measurable inverse. 

\begin{theo}Let $X$ be a measurable space. Then $Iso(X) = bij(X)$ if and only if $X$ is cocardinal.
\end{theo}

As a consequence, we have the following:

\begin{coro}Let $X$ be a measurable space. Then $Mor(X,X) = Map(X,X)$ if and only if $X$ has the trivial $\sigma$-algebra or the discrete $\sigma$-algebra.    
\end{coro}

Here $Mor(X,X)$ is the set of measurable maps from $X$ to $X$.

\begin{acknowledgements}
The \ author \ was \ partially \ supported \ by \ the joint \ FAPEMIG/CNPq program "Programa de Apoio à Fixação de Jovens Doutores no Brasil", Grant Number BPD-00721-22 (FAPEMIG), 150784/2023-6 (CNPq).
\end{acknowledgements}

\section{Proof of Theorem 1}

For a topological space $X$, consider $Closed(X)$ as the set of closed sets of $X$. For a set $X$, we denote the cardinality of $X$ by $\# X$.

\begin{lema}\label{samecardinalityisclosed}Let $X$ be a topological space such that $Homeo(X) = bij(X)$. Let $F \in Closed(X)$ such that $\# F < \# X$. If $F' \subseteq X$ such that $\# F = \# F'$, then $F' \in Closed(X)$.
\end{lema}

\begin{proof}If $X$ is infinite and $\# F < \# X$, then $\# (X - F) = \# X$. Analogously,  $\# (X - F') = \# X$. So  $\# (X - F) = \# (X - F')$ (this occurs trivially for a finite set $X$). So there exists a bijection $f: X \rightarrow X$ such that $f(F) = F'$ and $f(X- F) = X - F'$. Since $Homeo(X) = bij(X)$, then $f$ is a homeomorphism, which implies that $F'$ is closed.
\end{proof}

\begin{lema}\label{lesscardinalityisclosed}Let $X$ be a topological space such that $Homeo(X) = bij(X)$. Let $F \in Closed(X)$ such that $\# F < \# X$. If $F' \subseteq X$ such that $\# F > \# F'$, then $F' \in Closed(X)$.
\end{lema}

\begin{proof}Suppose that $F$ is finite.

Let $x \in F$. Then $\{x\} = \bigcap \{G \subseteq X: x \in G, \# G = \# F\}$. By the lemma above, all subsets $G$ are closed, which implies that $\{x\}$ is closed and, since $F$ is finite, every subset of $F$ is also closed.

Suppose that $F$ is infinite.

Let $G \subseteq F$ such that $\# G =\# F'$. Let $x \in F - G$. We have that $\# (F- \{x\}) = \# F$. By the lemma above, $F- \{x\}$ is closed. But $G = \bigcap_{x \in F-G} F-\{x\}$, which implies that $G$ is closed. Since $\# G = \# F'$, then, by the lemma above, $F'$ is closed.
\end{proof}

\begin{lema}\label{twosets}Let $X$ be an infinite set,  $x \in X$ and $F \subset X$ such that $\# F = \# X$. Then there exists  $U,V \subseteq X$ such that $U\cap V = \{x\}$, $\#U = \# V = \# (X-F)$ and $\#(X-U) = \# (X-V) = \# F$.    
\end{lema}

\begin{proof} Suppose that $ \#(X-F) < \aleph_{0}$. 

Since $X$ is infinite, there exists $U'$ and $V'$ disjoint subsets of $X$ such that $x \notin U' \cup V'$ and $\# U' = \# V' = \#(X-F)-1$. 

Suppose that $\aleph_{0} \leqslant \#(X-F) < \# X$. 

Since $X$ is infinite, there exists $U'$ and $V'$ disjoint subsets of $X$ such that $\# U' = \# V' = \#(X-F)$. We have that $\# (X-U') = \# X = \# F$ (since $\#(X-F) < \# X$) and, analogously,  $\# (X-V') = \# X = \# F$.

Suppose that $\#(X-F) = \# X$.

Take $U'$ and $V'$ disjoint subsets of $X$ such that $U' \cup V' = X-(F\cup\{x\})$ and $\# U' = \# V' = \# (X-F)$. We have that $\# (X-U') = \# F$ (since $F \subset X-U'$ and $\# F = \# X$) and, analogously, $\# (X-V') = \# F$.
 
In all cases, take $U = U' \cup \{x\}$ and $V = V' \cup \{x\}$. We have that $\# U = \# V = \#(X-F)$, $U \cap V = \{x\}$ and $\# (X-U) = \# (X-V) = \# X = \# F$.
\end{proof}

\begin{teo}\label{thm1}Let $X$ be a topological space. Then $Homeo(X) = bij(X)$ if and only if $X$ is cocardinal.
\end{teo}

\begin{proof}We already know that, for cocardinal spaces, $Homeo(X) = bij(X)$ holds.

Let $\kappa = \sup \{\kappa': \exists F \in Closed (X): F  \neq X, \# F = \kappa' \}$. 

Suppose that $\kappa = 0$.

If $F$ is closed in $X$ with $F \neq X$, then $\# F = 0$. Then $X$ has the trivial topology.

Suppose that $\kappa > 0$.

Suppose that $X$ is finite.

There exists $F \in Closed(X)$ such that $F \neq X$ and $\# F = \kappa$. By \textbf{Lemma \ref{lesscardinalityisclosed}}, every point is a closed set. Since $X$ is finite, it follows that $X$ has the discrete topology.

Suppose that $X$ is infinite.

Suppose that $\kappa \neq \# X$ or $\max \{\kappa'\!:\! \exists F\! \in\! Closed (X): F  \neq X, \# F = \kappa' \} \neq \# X$.
 
Let $\alpha < \kappa$. There exists cardinal $\kappa'$ such that $\alpha  \leqslant \kappa' \leqslant \kappa$ and there exists $F \in Closed(X)$ such that $F \neq X$ and $\# F = \kappa'$. By \textbf{Lemma \ref{lesscardinalityisclosed}}, every subset of $X$ with cardinality $\alpha$ is a closed set in $X$. If there exists a closed set $G$ with cardinality $\kappa$, then $\kappa \neq \# X$, by the assumptions on $\kappa$. Then every subset of $X$ with cardinality $\kappa$ is closed (by \textbf{Lemma \ref{samecardinalityisclosed}}). Since there is no closed set $G \neq X$ with cardinality bigger than $\kappa$, it follows that $X$ has the co-$\kappa^{+}$ topology. If there is no closed set different from $X$ with cardinality $\kappa$, it follows that $X$ has the co-$\kappa$ topology.

Suppose that $\kappa = \max \{\kappa': \exists F \in Closed (X): F  \neq X, \# F = \kappa' \} = \# X$.

There exists $F \in Closed(X)$ such that $F \neq X$ and $\# F = \# X$. Let $x \in X$. By \textbf{Lemma \ref{twosets}}, there exists $U_{x},V_{x} \subseteq X$ such that $U_{x}\cap V_{x} = \{x\}$, $\#U_{x} = \# V_{x} = \# (X-F)$ and $\#(X-U_{x}) = \# (X-V_{x}) = \# F$. Then there exists $f,g \in bij(X)$ such that $f(U_{x}) = X-F$ and $g(V_{x}) = X-F$. Since $bij(X) = Homeo(X)$, $U_{x}$ and $V_{x}$ are open sets, which implies that $\{x\}$ is an open set as well. Then $X$ has the discrete topology.

In all cases we get that $X$ is cocardinal.
\end{proof}

\begin{cor}\label{firstcorollary}Let $X$ be a topological space. Then $\mathcal{C}(X,X) = Map(X,X)$ if and only if $X$ has the trivial topology or the discrete topology.    
\end{cor}

\begin{proof}It is easy to see that if $X$ has the trivial topology or the discrete topology, then  $\mathcal{C}(X,X) = Map(X,X)$. 

Let $X$ with the co-$\kappa$ topology, for $\kappa$ a cardinal with $\aleph_{0} \leqslant \kappa < \# X^{+}$. Then there exists $F \subset X$ that is not closed. Let $x,y \in X$, with $x\neq y$. In this case, take a map $f: X \rightarrow X$ defined as $f(F) = x$, $f(X-F) = y$. Since $X$ has the co-$\kappa$ topology, with $\kappa > 1$, then $\{x\} \in Closed(X)$. But $F = f^{-1}(x)$, which implies that $f$ is not continuous. Thus $\mathcal{C}(X,X) \neq Map(X,X)$.

Since $\mathcal{C}(X,X) = Map(X,X)$ implies  $Homeo(X) = bij(X)$, then this corollary follows from the theorem above.
\end{proof}

\section{Proof of Theorem 2}

\begin{lema}\label{samecardinalityismeasurable}Let $(X,\Sigma)$ be a measurable space such that $Iso(X) = bij(X)$. Let $F \in \Sigma$ such that $\# F < \# X$. If $F' \subseteq X$ such that $\# F = \# F'$, then $F' \in \Sigma$.
\end{lema}

\begin{proof}Analogous to \textbf{Lemma \ref{samecardinalityisclosed}}.
\end{proof}

\begin{lema}\label{lesscardinalityismeasurable}Let $(X,\Sigma)$ be a measurable space such that $Iso(X) = bij(X)$. Let $F \in \Sigma$ such that $\# F < \# X$. If $F' \subseteq X$ such that $\# F > \# F'$, then $F' \in \Sigma$.
\end{lema}

\begin{proof}Suppose that $F$ is finite.

Let $x \in F$ and $y \in X-F$. Then $\{x\} = \bigcap_{z \in F-\{x\}} (F-\{z\})\cup\{y\} $. By the lemma above, for every $z \in F-\{x\}$, the set $(F-\{z\})\cup\{y\} \in \Sigma$ , which implies that $\{x\} \in \Sigma$ and, since $F$ is finite, every subset of $F$ is also in $\Sigma$. By the lemma above, we get that $F' \in \Sigma$.

Suppose that $F$ is infinite.

Let $G \subset F$ such that $\# G =\# F'$. We have that \#(F - G) = \#F. Take $U'$ and $V'$ disjoint subsets of $F$ such that $\# U' = \# V' = \# F$ and $U'\cup V' = F- G$. Take $U = U' \cup G$ and $V = V' \cup G$. We have that $\# U = \# V = \# F$. So, by the lemma above, $U,V \in \Sigma$. But $U\cap V = G$, which implies that $G \in \Sigma$. Since $\# G = \# F'$, then, by the lemma above, $F' \in \Sigma$.
\end{proof}

\begin{lema}\label{secondlesscardinalityismeasurable}Let $(X,\Sigma)$ be a measurable space such that $Iso(X) = bij(X)$. Let $F \in \Sigma$ such that $\# F = \# X$ and $\# (X-F) = \# X$. If $F' \subseteq X$ such that $\# F > \# F'$, then $F' \in \Sigma$.
\end{lema}

\begin{proof}Take $G \subseteq F$ such that $\# G = \# F'$. Take $G' = (X-F)\cup G$. We have that $\#G'= \#(X- G') = \# X$. Take a bijection $f: X \rightarrow X$ such that $f(F) = G'$ and $f(X-F) = X-G'$. Since $Iso(X) = bij(X)$, then $G' \in \Sigma$, which implies that $G = F \cap G' \in \Sigma$. 
\end{proof}

\begin{teo}Let $(X,\Sigma)$ be a measurable space. Then $Iso(X) = bij(X)$ if and only if $X$ is cocardinal.
\end{teo}

\begin{proof}It is easy to see that, for cocardinal spaces, $Iso(X) = bij(X)$ holds.

Let $\kappa = \sup \{\kappa' \neq \# X: \exists F \in \Sigma: \# F = \kappa' \}$. 

Suppose that $\kappa = 0$.

Let $F \in \Sigma$ with $F \neq X$. If $\# F > 0$, then $\# F = \# X$. Analogously, $\#(X-F) = \# X$. By the previous lemma, if $x \in X$, then $\{x\} \in \Sigma$, contradicting the fact that $\kappa = 0$. So $\# F = 0$. Then $X$ has the trivial $\sigma$-algebra.

Suppose that $\kappa > 0$.

Suppose that $X$ is finite.

There exists $F \in \Sigma$ such that $F \neq X$ and $\# F = \kappa$. By \textbf{Lemma \ref{lesscardinalityismeasurable}}, every point is a measurable set. Since $X$ is finite, it follows that $X$ has the discrete $\sigma$-algebra.

Suppose that $X$ is infinite.

Let $\alpha < \kappa$. There exists a cardinal $\kappa'$ such that $\alpha  \leqslant \kappa' \leqslant \kappa$ (or $\alpha  \leqslant \kappa' < \kappa$, if $\kappa = \# X$) and there exists $F \in \Sigma$ such that $\# F = \kappa'$. By \textbf{Lemma \ref{lesscardinalityismeasurable}}, every subset of $X$ with cardinality $\alpha$ is a measurable set in $X$. 

Suppose that $\kappa \neq \# X$.

If there exists a measurable set $G$ with cardinality $\kappa$, then every subset of $X$ with cardinality $\kappa$ is measurable (by the fact that $\kappa \neq \# X$ and \textbf{Lemma \ref{samecardinalityismeasurable}}). Let $G$ be a measurable set with cardinality bigger than $\kappa$. By the definition of $\kappa$, $\# G = \# X$. 

Suppose \ that \ there \ is \ no \ measurable \ set \ $G$ \ with \ $\# G = \# X$ \ and \ $\# (X-G) = \# X$. If there is no measurable set with cardinality $\kappa$, then $X$ has the $\kappa$-co-$\kappa$ algebra. If there is a measurable set with cardinality $\kappa$, then $X$ has the $\kappa^{+}$-co-$\kappa^{+}$ algebra.

Suppose that there is a measurable set $G$ with $\# G = \# X$ and $\# (X-G) = \# X$. Let $\beta$ be a cardinal such that $\kappa \leqslant \beta < \# X$. Take $G' \subseteq G$ such that $\# G' = \beta$. By \textbf{Lemma \ref{secondlesscardinalityismeasurable}}, $G' \in \Sigma$. By the definition of $\kappa$, we get that $\beta = \kappa$ and then $\# X = \kappa^{+}$. We also get that all subsets of $X$ with cardinality $\kappa$ are measurable. Let $H \subset X$ with $\# H > \kappa$. Then $\# H = \# X$. If $\#(X- H) \leqslant \kappa$, then $X-H \in \Sigma$, which implies that $H \in \Sigma$. If $\#(X- H) = \# X $, then take a bijection $g: X \rightarrow X$ such that $g(G) = H$ and $g(X-G) = X-H$. Since $Iso(X) = bij(X)$, then $H \in \Sigma$. So we get that $X$ has the discrete $\sigma$-algebra. 

Suppose that $\kappa = \# X$.

Suppose \ that \ there \ is \ no \ measurable \ set \ $G$ \ with \ $\# G = \# X$ \ and \ $\# (X-G) = \# X$. Since for every $\alpha < \kappa$ and every $H \subset X$ with $\# H = \alpha$, $H \in \Sigma$, then $X$ has the $\kappa$-co-$\kappa$ algebra.

Suppose that there is a measurable set $G$ with $\# G = \# X$ and $\# (X-G) = \# X$. If $H$ is another subset of $X$ with $\# H = \# X$ and $\# (X-H) = \# X$, then $H \in \Sigma$, by the same argument as in the case $\kappa \neq \# X$. So every subset of $X$ is in $\Sigma$, i. e., $X$ has the discrete $\sigma$-algebra.

In all cases we get that $X$ is cocardinal.
\end{proof}

\begin{cor}Let $X$ be a measurable space. Then $Mor(X,X) = Map(X,X)$ if and only if $X$ has the trivial $\sigma$-algebra or the discrete $\sigma$-algebra.    
\end{cor}

\begin{proof}Analogous to \textbf{Corollary \ref{firstcorollary}}.
\end{proof}

\end{document}